\documentclass[dvips]{amsart}

\usepackage[foot]{amsaddr}
\title[Categorical Verlinde formula]
{Recent developments of the categorical Verlinde formula}
\author[K.~Shimizu]{Kenichi Shimizu}
\email{kshimizu@shibaura-it.ac.jp}
\address{Department of Mathematical Sciences \\
  Shibaura Institute of Technology \\
  307 Fukasaku, Minuma-ku, Saitama-shi, Saitama 337-8570, Japan.}
\date{}

\usepackage{amsmath,amssymb,amscd}
\usepackage{upgreek}
\usepackage{bbm}
\usepackage[all,knot]{xy}

\usepackage{amsthm}
\numberwithin{equation}{section}

\newtheorem{counter}{}[section]

\theoremstyle{definition}
\newtheorem{definition}         [counter]{Definition}

\newtheorem*{notation*}         {Notation}

\theoremstyle{plain}
\newtheorem{lemma}              [counter]{Lemma}

\newtheorem{theorem}            [counter]{Theorem}

\newtheorem*{theorem*}          {Theorem}
\theoremstyle{remark}
\newtheorem{remark}             [counter]{Remark}
\newtheorem{example}            [counter]{Example}

\newcommand{\id}{\mathrm{id}}
\newcommand{\eval}{\mathrm{eval}}
\newcommand{\coev}{\mathrm{coev}}

\newcommand{\op}{\mathrm{op}}

\newcommand{\unitobj}{\mathbbm{1}}

\newcommand{\Vect}{\mathcal{V}ec}
\DeclareMathOperator{\Hom}{\mathrm{Hom}}
\DeclareMathOperator{\Img}{\mathrm{Im}}
\DeclareMathOperator{\SL}{\mathrm{SL}}
\DeclareMathOperator{\End}{\mathrm{End}}
\DeclareMathOperator{\Rep}{\mathrm{Rep}}

\DeclareMathOperator{\Gr}{\mathrm{Gr}}
\DeclareMathOperator{\CF}{\mathrm{CF}}
\DeclareMathOperator{\FPdim}{\mathrm{FPdim}}

\begin{document}

\begin{abstract}
  I review recent developments of `non-semisimple' modular tensor categories in the sense of Lyubashenko and the categorical Verlinde formula for such categories (this is a proceedings article for {\em Meeting for Study of Number theory, Hopf algebras and related topics} held at University of Toyama, Japan, 12--15 February 2017).
\end{abstract}

\maketitle

\section{Introduction}

Let $k$ be an algebraically closed field and, to avoid technical difficulty, assume that $k$ is of characteristic zero in Introduction. A {\em fusion category} is a $k$-linear semisimple rigid monoidal category, with only finitely many isomorphism classes of simple objects, such that the tensor unit $\unitobj \in \mathcal{C}$ is a simple object. Let $\mathcal{C}$ be a fusion category equipped with a ribbon structure (which allows us to interpret link diagrams colored with objects of $\mathcal{C}$ as morphisms in $\mathcal{C}$), and let $\{ V_i \}_{i \in I}$ be a complete set of representatives of isomorphism classes of simple objects of $\mathcal{C}$. Then the S-matrix of $\mathcal{C}$ is defined to be the square matrix $S = (s_{i j})_{i, j \in I}$, where
\begin{equation*}
  \SelectTips{cm}{10}
  s_{i j} =
  \knotholesize{10pt}
  \xy /r1.5pc/:
  (0,1) \vunder, (0,0) \vunder-,
  (0,1) \hcap[-2] >, (1,1) \hcap[2] >,
  (-1.5,0) *+!{V_i},
  (+2.5,0) *+!{V_j,} \endxy
\end{equation*}
that is, the invariant of the Hopf link colored with $V_i$ and $V_j$. A {\em modular tensor category} is a ribbon fusion category whose S-matrix is invertible. For fusion categories and modular tensor categories, see \cite{MR1797619,MR3242743}

The name `modular' comes from the fact that a modular tensor category yields a projective representation of the modular group $\SL_2(\mathbb{Z})$ \cite[Remark 3.1.2]{MR1797619}: We recall that the modular group $\SL_{2}(\mathbb{Z})$ is generated by
\begin{equation*}
  \mathfrak{s} := \begin{pmatrix}
    0 & -1 \\ 1 & 0
  \end{pmatrix}
  \quad \text{and} \quad
  \mathfrak{t} := \begin{pmatrix}
    1 & 1 \\ 0 & 1
  \end{pmatrix}
\end{equation*}
with relations $(\mathfrak{s} \mathfrak{t})^3 = \mathfrak{s}^2$ and $\mathfrak{s}^4 = 1$. Let $\mathcal{C}$ be a modular tensor category over $\mathbb{C}$, and let $\{ V_i \}_{i \in I}$ be a complete set of representatives of isomorphism classes of simple objects of $\mathcal{C}$. Then the assignment
\begin{equation*}
  \mathfrak{s} \mapsto (s_{i j})_{i, j \in I},
  \quad
  \mathfrak{t} \mapsto (\delta_{i j} \theta_i)_{i, j \in I}
\end{equation*}
defines a projective representation of $\SL_2(\mathbb{Z})$, where $\theta_i$ is the scalar such that the twist of $\mathcal{C}$ is equal to $\theta_i \id_{V_i}$ on the simple object $V_i$.

This projective representation of $\SL_2(\mathbb{Z})$ is in fact a part of the 3-dimensional topological quantum field theory arising from a modular tensor category. In particular, a modular tensor category yields an invariant of oriented closed 3-manifolds and a projective representation of the mapping class group of any closed oriented surface; see \cite{MR1797619} and \cite{MR1292673}. The action of $\SL_2(\mathbb{Z})$ is just the case where the surface is a torus.

Another important result on modular tensor categories is the {\em categorical Verlinde formula} that describes the decomposition rule of the tensor product. We keep the above notation and suppose that we have $[V_i \otimes V_j] = \sum_{k \in I} N_{i j}^k [V_k]$ ($i, j \in I$) in the Grothendieck ring of $\mathcal{C}$. The formula states
\begin{equation}
  \label{eq:cat-Ver-ss}
  N_{i j}^a = \frac{1}{\dim(\mathcal{C})} \sum_{r \in I} \frac{s_{i r} s_{j r} s_{a^* r}}{\dim(V_r)}
  \quad (i, j, a \in I),
\end{equation}
where $a^* \in I$ for $a \in I$ is the label such that $V_{a^*}$ is dual to $V_a$,
\begin{equation*}
  \dim(V_r) = \SelectTips{cm}{10} \xy /r1pc/:
  (0,1) \hcap[-2] >, (0,1) \hcap[2], (1.75,0) *+!{V_r}
  \endxy
  \quad (=\text{the quantum dimension of $V_r$}),
\end{equation*}
and $\dim(\mathcal{C}) = \sum_{r \in I} \dim(V_r)^2$.

It is interesting to drop the semisimplicity assumption from the definition of modular tensor categories. Lyubashenko \cite{MR1324034,MR1352517,MR1324033,MR1354257} has introduced such a `non-semisimple' generalization of modular tensor categories and showed that a `non-semisimple' modular tensor category also yields a topological invariants of closed 3-manifolds and a projective representation of $\SL_2(\mathbb{Z})$. This notion can be a nice categorical framework to study representation theory of factorizable Hopf algebras such as so-called small quantum groups. From a physical point of view, `non-semisimple' modular tensor categories are also expected that they are useful in the study of logarithmic conformal field theories; see, {\it e.g.},  \cite{MR3146014,2016arXiv160504630C,2016arXiv160504448G,2017arXiv170201086F,2017arXiv170300150G,2017arXiv170608164F}.

This article reviews algebraic aspects of such `non-semisimple' modular tensor categories. The organization is as follows: Lyubashenko's notion of a modular tensor category is defined as a `non-degenerate' braided finite tensor category equipped with a ribbon structure. We first explain what does `non-degenerate' means in Section~\ref{sec:mtc}, and then give several characterization of the non-degeneracy in Section~\ref{sec:characterizations} following \cite{2016arXiv160206534S}. The results of \cite{2016arXiv160206534S} is useful to provide examples of modular tensor categories; see Section~\ref{sec:examples}. Finally, in Section~\ref{sec:verlinde-formula}, we review a non-semisimple version of the categorical Verlinde formula recently proposed by Gainutdinov and Runkel in \cite{2016arXiv160504448G}.

\subsection*{Acknowledgment}

The author would like to thank the organizers of {\em Meeting for Study of Number theory, Hopf algebras and related topics}. The author is supported by JSPS KAKENHI Grant Number 16K17568.

\section{Modular tensor categories}
\label{sec:mtc}

\subsection{Monoidal categories}

We fix our convention for monoidal categories and related notions \cite{MR1712872,MR3242743}. In view of Mac Lane's coherence theorem, we assume that all monoidal categories are strict. Let $\mathcal{C}$ be a monoidal category with tensor product $\otimes$ and tensor unit $\unitobj$. Let $L$ and $R$ be objects of $\mathcal{C}$, and let $\varepsilon: L \otimes R \to \unitobj$ and $\eta: \unitobj \to R \otimes L$ be morphisms in $\mathcal{C}$. We say that $(L, \varepsilon, \eta)$ is a {\em left dual object} of $R$ and $(R, \varepsilon, \eta)$ is a {\em right dual object} of $\mathcal{C}$ if the equations
\begin{equation*}
  (\varepsilon \otimes \id_L) \circ (\id_L \otimes \eta) = \id_L
  \quad \text{and} \quad
  (\id_R \otimes \varepsilon) \circ (\eta \otimes \id_R) = \id_R
\end{equation*}
hold. The monoidal category $\mathcal{C}$ is said to be {\rm rigid} if every object of $\mathcal{C}$ has a left dual object and a right dual object.

Now we suppose that $\mathcal{C}$ is a rigid monoidal category. Given an object $X \in \mathcal{C}$, we denote by $(X^*, \eval_X: X^* \otimes X \to \unitobj, \coev_X: \unitobj \to X \otimes X^*)$ the (fixed) left dual object of $X$. We note that the assignment $X \mapsto X^*$ is a contravariant endofunctor on $\mathcal{C}$ that preserves the tensor unit and reverses the order of the tensor product.

\subsection{Finite tensor categories}

By a $k$-algebra, we always mean an associative and unital algebra over $k$. Given a $k$-algebra $A$, we denote by $\Rep(A)$ the category of finite-dimensional left $A$-modules. A {\em finite abelian category} (over $k$) is a $k$-linear category that is $k$-linearly equivalent to $\Rep(A)$ for some finite-dimensional $k$-algebra $A$. Of course, such an algebra $A$ is determined only up to Morita equivalence. An intrinsic definition of a finite abelian category is found in, for example, \cite[Section 1.8]{MR3242743}.

\begin{definition}
  A {\em finite tensor category} \cite{MR2119143} is a rigid monoidal category $\mathcal{C}$ such that $\mathcal{C}$ is a finite abelian category, the tensor product of $\mathcal{C}$ is $k$-bilinear, and the tensor unit of $\mathcal{C}$ is a simple object. A {\em fusion category} \cite{MR2183279} is a synonym for a semisimple finite tensor category.
\end{definition}

Let $\mathcal{C}$ be a finite tensor category. By a {\em tensor full subcategory} of $\mathcal{C}$, we mean a non-zero full subcategory of $\mathcal{C}$ closed under taking finite direct sums, subquotients, tensor products and duals. We note that a tensor full subcategory of a finite tensor category is also a finite tensor category such that the inclusion functor preserves and reflects exact sequences.

\subsection{Coends}

Let $\mathcal{C}$ and $\mathcal{V}$ be categories, and let $H$ be a functor from $\mathcal{C}^{\op} \times \mathcal{C}$ to $\mathcal{V}$. A {\em coend} of $H$ is an object $F \in \mathcal{V}$ equipped with a family
\begin{equation*}
  \iota = \{ \iota_X: H(X, X) \to F \}_{X \in \mathcal{C}}
\end{equation*}
of morphisms in $\mathcal{V}$ parametrized by objects of $\mathcal{C}$ such that the following two conditions are satisfied:
\begin{enumerate}
\item The family $\iota$ is a {\em dinatural transformation} from $H$ to $F$ meaning that, for all morphisms $f: X \to Y$ in $\mathcal{C}$,  the following diagram commutes:
  \begin{equation*}
    \xymatrix@R=5pt@C=56pt{
      & H(X, X) \ar[rd]^{\iota_X} \\
      H(Y, X)
      \ar[ru]^{H(f, \id_X)}
      \ar[rd]_{H(\id_Y, f)}
      & & F. \\
      & H(Y, Y) \ar[ru]_{\iota_Y}
    }
  \end{equation*}
\item For every dinatural transformation $\iota' = \{ \iota'_X: H(X, X) \to F' \}_{X \in \mathcal{C}}$ from $H$ to an object $F' \in \mathcal{V}$, there exists a unique morphism $\phi: F \to F'$ in $\mathcal{V}$ such that $\phi \circ \iota_X = \iota'_X$ for all objects $X \in \mathcal{C}$ (the universal property).
\end{enumerate}
If it exists, a coend of $H$ is unique up to canonical isomorphism by the universal property. Following \cite{MR1712872}, we denote by
\begin{equation*}
  \int^{X \in \mathcal{C}} H(X, X)
\end{equation*}
the coend of $H$. See \cite{MR1712872} for more details.

\subsection{Modular tensor categories}

Let $\mathcal{C}$ be a finite tensor category. We have a functor $\mathcal{C}^{\op} \times \mathcal{C} \to \mathcal{C}$ given by $(X, Y) \mapsto X^* \otimes Y$. After technical arguments, we see that the coend of this functor,
\begin{equation*}
  \mathbb{F}_{\mathcal{C}} = \int^{X \in \mathcal{C}} X^* \otimes X,
\end{equation*}
exists. We write $\mathbb{F} = \mathbb{F}_{\mathcal{C}}$ if there is no confusion.

The coend $\mathbb{F}$ is a coalgebra in $\mathcal{C}$ by the structure morphisms defined by using the universal property. If $\mathcal{C}$ has a braiding $\sigma_{X,Y}: X \otimes Y \to Y \otimes X$ ($X, Y \in \mathcal{C}$), then, with the help of the Fubini theorem for coends, one can make the coalgebra $\mathbb{F}$ a Hopf algebra in $\mathcal{C}$. Instead of giving the detailed description of the structure morphisms of $\mathbb{F}$, we remark that the construction of $\mathbb{F}$ can be regarded as a part of Tannaka theory: We construct a Hopf algebra from a fiber functor in Tannaka theory, but this construction works for more general class of tensor functors. The Hopf algebra $\mathbb{F}$ can be obtained by applying such a general construction to $\id_{\mathcal{C}}$.

\begin{definition}
  \label{def:br-FTC-non-degen}
  Let $\mathcal{C}$ be a braided finite tensor category with braiding $\sigma$, and let $\mathbb{F}$ be as above. We define $\omega: \mathbb{F} \otimes \mathbb{F} \to \mathbb{F}$ to be the unique morphism such that, for all objects $X, Y \in \mathcal{C}$, the following equation holds:
  \begin{equation*}
    \omega \circ (i_X \otimes \iota_Y) = (\eval_X \otimes \eval_Y)
    \circ (\id_{X^*} \otimes \sigma_{Y^*,X}\sigma_{X,Y^*} \otimes \id_Y).
  \end{equation*}
  The braided finite tensor category $\mathcal{C}$ is {\em non-degenerate} if $\omega$ is non-degenerate in the sense that the composition
  \begin{equation*}
    \mathbb{F} = \mathbb{F} \otimes \unitobj
    \xrightarrow{\quad \id \otimes \coev \quad}
    \mathbb{F} \otimes \mathbb{F} \otimes \mathbb{F}^*
    \xrightarrow{\quad \omega \otimes \id \quad}
    \unitobj \otimes \mathbb{F}^* = \mathbb{F}^*
  \end{equation*}
  is an isomorphism in $\mathcal{C}$. A {\em modular tensor category} is a non-degenerate braided finite tensor category equipped with a twist.
\end{definition}

Lyubashenko has considered a more general setting than the above in his pioneering works \cite{MR1324034,MR1352517,MR1324033,MR1354257}. The above definition coincides with the notion of a {\em perfect modular category} of \cite{MR1354257}. The above definition can be found as \cite[Definition 5.2.7]{MR1862634} in the book of Kerler and Lyubashenko.

Let us clarify the usage of terminology: From now on, the term `modular tensor category' is always used in the above sense. We warn that it is {\em not} trivial that this definition agrees with the definition by the S-matrix in the semisimple case. We will discuss several `non-degeneracy' conditions in the next section. Consequently, a ribbon fusion category with invertible S-matrix is the same thing as a semisimple modular tensor category.

A modular tensor category yields an invariant of closed 3-manifolds and a projective representation of the mapping class group of a surface \cite{MR1324034,MR1352517,MR1324033,MR1354257,MR1862634}. Thus modular tensor categories are an interesting class of categories as a framework to construct topological invariants. Considering the fact that several problems in semisimple Hopf algebras have been solved by using semisimple modular tensor categories, we also expect that modular tensor categories are essential for representation theory of non-semisimple Hopf algebras. 

\subsection{Example: Factorizable Hopf algebras}

We examine what the definition of a modular tensor category says in the case where $\mathcal{C} = \Rep(H)$ for some finite-dimensional Hopf algebra $H$. If this is the case, then $\mathbb{F} = H^*$ (the dual space of $H$) as a vector space. The action $\triangleright$ of $H$ on $\mathbb{F}$ is given by
\begin{equation*}
  (h \triangleright f)(h') = f(S(h_{(1)}) h' h_{(2)})
\end{equation*}
for $h, h' \in H$ and $f \in H^*$, where $S$ is the antipode of $H$ and $h_{(1)} \otimes h_{(2)}$ is the comultiplication of $h$ in the Sweedler notation. The comultiplication and the counit of $\mathbb{F}$ are the same structure morphisms as the dual Hopf algebra of $H$.

Now we suppose that $H$ has a universal R-matrix $R \in H \otimes H$ so that $\mathcal{C}$ is a braided finite tensor category. The Hopf algebra structure of $\mathbb{F}$ and the paring $\omega$ are described explicitly in \cite[7.4.6]{MR1862634}. According, $\omega$ is given by
\begin{equation*}
  \omega(f \otimes g) = \sum_{i, j} f(b_{i} a_{j}) g(S(a_{i} b_{j})),
\end{equation*}
where $R = \sum_i a_i \otimes b_i$. Thus the braided finite tensor category $\mathcal{C}$ is non-degenerate if and only if the quasitriangular Hopf algebra $(H, R)$ is factorizable in the sense of \cite{MR1075721,MR1825894}. Summarizing the above argument, the notion of a modular tensor category may be thought of as a categorical generalization of a ribbon factorizable Hopf algebra.

\subsection{Properties of the coend $\mathbb{F}$}
\label{subsec:coend-F-property}

Let $\mathcal{C}$ be a finite tensor category, which may not be braided. We collect several basic properties of the coend $\mathbb{F}$.

\subsubsection{The coend $\mathbb{F}$ and the Drinfeld center}

We recall that the {\em Drinfeld center} of $\mathcal{C}$ is the category $\mathcal{Z}(\mathcal{C})$ defined as follows: An object of this category is a pair $(V, \xi)$ consisting of an object $V \in \mathcal{C}$ and a natural transformation
\begin{equation*}
  \xi_X: V \otimes X \to X \otimes V \quad (X \in \mathcal{C}),
\end{equation*}
called the {\em half-braiding}, such that the equations
\begin{equation*}
  \xi_{\unitobj} = \id_{V}
  \quad \text{and} \quad
  \xi_{X \otimes Y} = (\id_X \otimes \xi_Y) \circ (\xi_X \otimes \id_Y)
\end{equation*}
hold for all objects $X, Y \in \mathcal{C}$ (it is worth to note that the rigidity of $\mathcal{C}$ implies that $\xi$ is invertible). A morphism $(V, \xi) \to (W, \zeta)$ in $\mathcal{Z}(\mathcal{C})$ is a morphism from $V$ to $W$ in $\mathcal{C}$ compatible with the half-braidings $\xi$ and $\zeta$.

The Drinfeld center $\mathcal{Z}(\mathcal{C})$ closely relates to the coend $\mathbb{F}$ as follows: It is known that the forgetful functor $U: \mathcal{Z}(\mathcal{C}) \to \mathcal{C}$, given by $(V, \xi) \mapsto V$, has a left adjoint, say $L: \mathcal{C} \to \mathcal{Z}(\mathcal{C})$. It is also known that the functor
\begin{equation*}
  Z: \mathcal{C} \to \mathcal{C},
  \quad V \mapsto \int^{X \in \mathcal{C}} X^* \otimes V \otimes X
\end{equation*}
has a structure of a Hopf monad \cite{MR2355605,MR2869176,MR2793022} such that the category of $Z$-modules is isomorphic to $\mathcal{Z}(\mathcal{C})$. If we identify these categories, then $L$ is just the free $Z$-module functor. Hence we have
\begin{equation}
  \label{eq:F-U-L-1}
  \mathbb{F} = Z(\unitobj) = U L(\unitobj).
\end{equation}

\subsubsection{The adjoint algebra}

Let $U$ and $L$ be as above. It is known that $U$ also has a right adjoint, say $R$. We call $\mathbb{A} := U R(\unitobj)$ the {\em adjoint algebra} of $\mathcal{C}$ \cite{MR3631720}, since it generalizes the adjoint representation of a Hopf algebra. By \cite[Lemma 3.5]{MR2869176}, there are natural isomorphisms $L(V^*) \cong R(V)^*$ and $R(V^*) \cong L(V)^*$. Hence we have $\mathbb{F} \cong \mathbb{A}^*$ and $\mathbb{A} \cong \mathbb{F}^*$. In this sense, a result on the adjoint algebra can be translated into a result on the coend $\mathbb{F}$, and vice versa.

The adjoint algebra is more useful than $\mathbb{F}$ when we consider applications to the representation theory of Hopf algebras. For this reason, the author have formulated most of the results of \cite{MR3631720} in terms of the adjoint algebra. Thus, to apply results of \cite{MR3631720} to modular tensor categories, one may need to interpret them by the fact that the coend $\mathbb{F}$ is dual to $\mathbb{A}$.

\subsubsection{The space $\Hom_{\mathcal{C}}(\mathbb{F}, \unitobj)$}

Since $\mathbb{F}$ is a coalgebra in $\mathcal{C}$, the space $\Hom_{\mathcal{C}}(\mathbb{F}, \unitobj)$ is a $k$-algebra by the multiplication given by $a \star b := (a \otimes b) \circ \Delta$. Every object $X \in \mathcal{C}$ has a canonical right coaction $\rho_X: X \to X \otimes \mathbb{F}$ given by
\begin{equation}
  \label{eq:can-coact}
  \rho_X = \Big(
  X \xrightarrow{\quad \coev \otimes \id \quad} X \otimes X^* \otimes X
  \xrightarrow{\quad \id \otimes \iota_X \quad} X \otimes \mathbb{F} \Big),
\end{equation}
where $\iota_X: X^* \otimes X \to \mathbb{F}$ is the universal dinatural transformation. Given a morphism $a: \mathbb{F} \to \unitobj$, we define a natural transformation $\widetilde{a}: \id_{\mathcal{C}} \to \id_{\mathcal{C}}$ by
\begin{equation*}
  \widetilde{a}_X = (\id_X \otimes a) \circ \rho_X \quad (X \in \mathcal{C}).
\end{equation*}
The assignment $a \mapsto \widetilde{a}$ gives an isomorphism
\begin{equation}
  \label{eq:F-End-id}
  \Hom_{\mathcal{C}}(\mathbb{F}, \unitobj) \xrightarrow{\quad \cong \quad} \End(\id_{\mathcal{C}})
\end{equation}
of $k$-algebras; see \cite[Proposition 5.2.5]{MR1862634}.

\subsubsection{The space $\Hom_{\mathcal{C}}(\unitobj, \mathbb{F})$}

The above discussion means that $\Hom_{\mathcal{C}}(\mathbb{F}, \unitobj)$ may be identified with the `center' of an algebra. On the other hand, the vector space 
\begin{equation}
  \label{eq:def-CF}
  \CF(\mathcal{C}) := \Hom_{\mathcal{C}}(\unitobj, \mathbb{F})
\end{equation}
may be viewed as the space of `class functions' on a group. Indeed, if $\mathcal{C} = \Rep(H)$ for some finite-dimensional Hopf algebra $H$, then there is an isomorphism
\begin{equation}
  \label{eq:CF-Hopf}
  \CF(\mathcal{C}) \xrightarrow{\quad \cong \quad}
  C(H) := \{ f \in H^* \mid \text{$f(a b) = f(b S^2(a))$ for all $a, b \in H$} \}
\end{equation}
of vector spaces. We note that $C(H)$ is a $k$-algebra as a subalgebra of $H^*$. For an arbitrary finite tensor category $\mathcal{C}$, we have an isomorphism
\begin{equation*}
  \CF(\mathcal{C}) = \Hom_{\mathcal{C}}(\unitobj, \mathbb{F})
  = \Hom_{\mathcal{C}}(\unitobj, U L(\unitobj))
  \cong \Hom_{\mathcal{C}}(L(\unitobj), L(\unitobj))
\end{equation*}
of vector spaces by \eqref{eq:F-U-L-1}. Hence one can endow $\CF(\mathcal{C})$ with a structure of an algebra via this isomorphism; see \cite{MR3631720}, where the corresponding result for the adjoint algebra $\mathbb{A} = \mathbb{F}^*$ is explained in detail.

\subsection{Properties of the paring $\omega$}

We assume that $\mathcal{C}$ has a braiding. We have introduced the paring $\omega: \mathbb{F} \otimes \mathbb{F} \to \mathbb{F}$ to define the non-degeneracy of $\mathcal{C}$. This is in fact a {\em Hopf paring} in the sense that it is compatible with the structure morphisms of $\mathbb{F}$ in a certain way. Keeping this fact in mind, we consider the map
\begin{equation}
  \label{eq:br-FTC-Dri-map}
  \Phi_{\mathcal{C}}: \Hom_{\mathcal{C}}(\unitobj, \mathbb{F}) \to \Hom_{\mathcal{C}}(\mathbb{F}, \unitobj),
  \quad \chi \mapsto \omega \circ (\chi \otimes \id_{\mathbb{F}}).
\end{equation}
As we have seen, $\Hom_{\mathcal{C}}(\unitobj, \mathbb{F})$ and $\Hom_{\mathcal{C}}(\mathbb{F}, \unitobj)$ are $k$-algebras. Moreover, the multiplication $\star$ of $\Hom_{\mathcal{C}}(\unitobj, \mathbb{F})$ is given by $f \star g = m \circ (f \otimes g)$ under our assumption that $\mathcal{C}$ is braided \cite{MR3631720}, where $m$ is the multiplication of $\mathbb{F}$. The fact that $\omega$ is a Hopf paring implies that the map $\Phi_{\mathcal{C}}$ is in fact a homomorphism of $k$-algebras.

\section{Characterizations of non-degeneracy}
\label{sec:characterizations}

\subsection{Non-degeneracy conditions}

There are several characterizations of semisimple modularity of ribbon fusion categories; see, {\it e.g.}, \cite{MR1741269,MR1966525,MR2609644,MR3242743}. It is natural to attempt to extend these results to the non-semisimple setting. In this section, we introduce alternative `non-degeneracy' conditions for a braided finite tensor category (Definitions~\ref{def:br-FTC-factorizable},~\ref{def:br-FTC-Mueger} and~\ref{def:br-FTC-Dri-map}), and then give an outline of the proof of the equivalence of these conditions.

We first recall Schneider's characterization of factorizability: Let $H$ be a finite-dimensional quasitriangular Hopf algebra. The universal R-matrix $R$ of $H$ defines a Drinfeld twist $F$ of $H \otimes H$, and hence we obtain a new Hopf algebra $(H \otimes H)^{F}$ by twisting the comultiplication of $H \otimes H$ by $F$. Schneider introduced a Hopf algebra map $\delta_R: D(H) \to (H \otimes H)^F$ given in terms of the universal R-matrix, where $D(H)$ is the Drinfeld double of $H$, and then showed that $(H, R)$ is factorizable if and only if the map $\delta_R$ is bijective \cite[Theorem 4.3]{MR1825894}.

This result can be rephrased that $(H, R)$ is factorizable if and only if a certain functor $\Rep((H \otimes H)^F) \to \Rep(D(H))$ defined in terms of the braiding of $\Rep(H)$ is an equivalence. We note that there are equivalences
\begin{equation*}
  \Rep((H \otimes H)^F)
  \approx \Rep(H) \boxtimes \Rep(H)
  \quad \text{and} \quad
  \Rep(D(H))
  \approx \mathcal{Z}(\Rep(H))
\end{equation*}
of tensor categories, where $\boxtimes$ means the Deligne tensor product. Based on this observation, we have the following alternative `non-degeneracy' condition for braided finite tensor categories:

\begin{definition}[Etingof-Nikshych-Ostrik \cite{MR2097289}]
  \label{def:br-FTC-factorizable}
  Let $\mathcal{C}$ be a braided finite tensor category with braiding $\sigma$. We define $G: \mathcal{C} \boxtimes \mathcal{C} \to \mathcal{Z}(\mathcal{C})$ by the composition
  \begin{equation*}
    \mathcal{C} \boxtimes \mathcal{C} \xrightarrow{\quad T_{+} \boxtimes T_{-} \quad}
    \mathcal{Z}(\mathcal{C}) \boxtimes \mathcal{Z}(\mathcal{C})
    \xrightarrow{\quad \text{the tensor product} \quad} \mathcal{Z}(\mathcal{C}),
  \end{equation*}
  where $T_{+}(X) = (X, \sigma_{X,\, ?})$ and $T_{-}(X) = (X, \sigma_{?,X}^{-1})$. We say that $\mathcal{C}$ is {\em factorizable} if the functor $G$ is an equivalence.
\end{definition}

Another `non-degeneracy' condition is due to M\"uger \cite{MR1966524,MR1966525}.

\begin{definition}
  \label{def:br-FTC-Mueger}
  Let $\mathcal{C}$ be a braided finite tensor category with braiding $\sigma$. An object $T \in \mathcal{C}$ is {\em transparent} if $\sigma_{X,T}\sigma_{T,X} = \id_{T \otimes X}$ for all objects $X \in \mathcal{C}$. We define the {\em M\"uger center} of $\mathcal{C}$ to be the full subcategory $\mathcal{C}' \subset \mathcal{C}$ consisting of all transparent objects of $\mathcal{C}$. We say that {\em the M\"uger center of $\mathcal{C}$ is trivial} if every object of $\mathcal{C}'$ is isomorphic to the direct sum of finitely many copies of the tensor unit of $\mathcal{C}$.
\end{definition}

The next condition for a braided finite tensor category closely relates to the invertibility of the S-matrix of a ribbon fusion category.

\begin{definition}
  \label{def:br-FTC-Dri-map}
  Let $\mathcal{C}$ be a braided finite tensor category. We say that $\mathcal{C}$ is {\em weakly-factorizable} if the map $\Phi_{\mathcal{C}}$, given by~\eqref{eq:br-FTC-Dri-map}, is bijective ({\it cf}. Takeuchi \cite{MR1850651}; see also Remark~\ref{rem:w-factor} below).
\end{definition}

If $\mathcal{C}$ is a ribbon fusion category, then the map $\Phi_{\mathcal{C}}$ is represented by the S-matrix of $\mathcal{C}$ with respect to certain bases: Let $\{ V_i \}_{i \in I}$ be a complete set of representatives of isomorphism classes of simple objects of $\mathcal{C}$. By the semisimplicity, we may assume that $\mathbb{F} = \bigoplus_{i \in I} V_i^* \otimes V_i$ as an object of $\mathcal{C}$. For each $i \in I$, we define morphisms $\chi_i: \unitobj \to \mathbb{F}$ and $e_i: \mathbb{F} \to \unitobj$ in $\mathcal{C}$ by the composition
\begin{equation*}
  \newcommand{\xarr}[1]{\xrightarrow{\makebox[5em]{$\scriptstyle #1$}}}
  \begin{aligned}
    \chi_i & = \Big( \unitobj
    \xarr{\coev'} V_i^* \otimes V_i
    \xarr{\text{inclusion}} \mathbb{F} \Big), \\
    e_i & = \Big( \mathbb{F}
    \xarr{\text{projection}} V_i^* \otimes V_i
    \xarr{\eval} \unitobj \Big),
  \end{aligned}
\end{equation*}
where $\coev'_X = (\id_{X^*} \otimes \theta_X) \circ \sigma_{X,X^*} \circ \coev_X$. Then we have
\begin{equation*}
  \Phi_{\mathcal{C}}(\chi_i) = \sum_{j \in I} \frac{s_{i j}}{s_{0 j}} \cdot e_j
\end{equation*}
where $S = (s_{i j})$ is the S-matrix of $\mathcal{C}$; see, {\it e.g.}, \cite{MR1797619,MR1850651}. Thus the S-matrix of $\mathcal{C}$ is invertible if and only if the linear map $\Phi_{\mathcal{C}}$ is invertible.

\begin{theorem}[Shimizu \cite{2016arXiv160206534S}]
  \label{thm:non-degen}
  For a braided finite tensor category $\mathcal{C}$, the following conditions are equivalent:
  \begin{enumerate}
  \item $\mathcal{C}$ is non-degenerate in the sense of Definition~\ref{def:br-FTC-non-degen}.
  \item $\mathcal{C}$ is factorizable in the sense of Definition~\ref{def:br-FTC-factorizable}.
  \item The M\"uger center of $\mathcal{C}$ is trivial in the sense of Definition~\ref{def:br-FTC-Mueger}.
  \item $\mathcal{C}$ is weakly-factorizable in the sense of Definition~\ref{def:br-FTC-Dri-map}.
  \item The map $\Phi_{\mathcal{C}}$ is injective.
  \end{enumerate}
\end{theorem}

It has been known that these conditions are equivalent under the assumption that $\mathcal{C}$ is semisimple; see \cite{MR1741269,MR1966525,MR2609644,MR3242743}. Thus we have obtained a non-semisimple generalization of a well-known characterization of semisimple modular tensor categories.

\begin{remark}
  Suppose that $\mathcal{C} = \Rep(H)$ for some finite-dimensional quasitriangular Hopf algebra $H$. As we have recalled in Subsection~\ref{subsec:coend-F-property}, we may identify
  \begin{equation*}
    \Hom_{\mathcal{C}}(\unitobj, \mathbb{F}) = C(H)
    \quad \text{and} \quad
    \Hom_{\mathcal{C}}(\mathbb{F}, \unitobj) = Z(H),
  \end{equation*}
  where $C(H)$ is given in~\eqref{eq:CF-Hopf} and $Z(H)$ is the center of $H$. By this identification, we see that $H$ is weakly-factorizable in the sense of Takeuchi \cite{MR1850651} if and only if the map $\Phi_{\mathcal{C}}$ is bijective. The above theorem implies that $H$ is weakly-factorizable if and only if it is factorizable.
\end{remark}

\begin{remark}
  \label{rem:w-factor}
  Let $\mathcal{C}$ be a braided finite tensor category. We remark that the surjectivity of $\Phi_{\mathcal{C}}$ does not imply the non-degeneracy. An example is given by Sweedler's 4-dimensional Hopf algebra; see \cite[Remark 5.4]{2016arXiv160206534S}.

  For a finite tensor category $\mathcal{A}$, the {\em distinguished invertible object} $\rho \in \mathcal{A}$ \cite{MR2119143} is defined to be the socle of $P_0^*$, where $P_0$ is the projective cover of $\unitobj \in \mathcal{A}$. This object generalizes the modular function on a finite-dimensional Hopf algebra, and thus we say that $\mathcal{A}$ is {\em unimodular} if $\rho \cong \unitobj$. If $\mathcal{A}$ is unimodular, then the equation
  \begin{equation*}
    \dim_k \Hom_{\mathcal{A}}(\mathbb{F}_{\mathcal{A}}, \unitobj)
    = \dim_k \Hom_{\mathcal{A}}(\unitobj, \mathbb{F}_{\mathcal{A}})
  \end{equation*}
  holds \cite{MR3631720}. Hence, for a unimodular braided finite tensor category $\mathcal{C}$, the conditions of Theorem~\ref{thm:non-degen} are also equivalent to that the map $\Phi_{\mathcal{C}}$ is surjective.
\end{remark}

\subsection{Non-degeneracy and factorizability} 

We sketch the proof of the equivalence (1) $\Leftrightarrow$ (2) of Theorem~\ref{thm:non-degen}. Let $\mathcal{C}$ be a braided finite tensor category, and let $\mathbb{F}$ be the Hopf algebra used to define the non-degeneracy of $\mathcal{C}$. There are the category $\mathcal{C}_{\mathbb{F}}$ of right $\mathbb{F}$-modules and the category $\mathcal{C}^{\mathbb{F}}$ of right $\mathbb{F}$-comodules. A right $\mathbb{F}$-comodule $V$ is also a right $\mathbb{F}$-module by the action given by
\begin{equation*}
  V \otimes \mathbb{F}
  \xrightarrow{\quad \text{coaction} \otimes \id \quad}
  V \otimes \mathbb{F} \otimes \mathbb{F}
  \xrightarrow{\quad \id \otimes \omega \quad} V \otimes \unitobj = V,
\end{equation*}
since the paring $\omega: \mathbb{F} \otimes \mathbb{F} \to \mathbb{F}$ is in fact a Hopf paring. This correspondence defines a functor from $\mathcal{C}^{\mathbb{F}}$ to $\mathcal{C}_{\mathbb{F}}$, which we denote by $\omega^{\natural}: \mathcal{C}^{\mathbb{F}} \to \mathcal{C}_{\mathbb{F}}$. We note that $\mathcal{C}$ is non-degenerate if and only if $\omega^{\natural}$ is an equivalence.

Lyubashenko showed that the functor
\begin{equation}
  \label{eq:equiv-Lyu}
  \mathcal{C} \boxtimes \mathcal{C} \to \mathcal{C}^{\mathbb{F}},
  \quad W \boxtimes X \mapsto (W \otimes X, \id_W \otimes \rho_X)
\end{equation}
is an equivalence of categories, where $\rho_X: X \to X \otimes \mathbb{F}$ is the canonical coaction given by~\eqref{eq:can-coact}. On the other hand, the category $\mathcal{C}^{\mathbb{F}}$ has the following description: Let $M$ be a right $\mathbb{F}$-module with action $a_M: M \otimes \mathbb{F} \to M$. We define the natural transformation $\xi: M \otimes \id_{\mathcal{C}} \to \id_{\mathcal{C}} \otimes M$ by the composition
\begin{equation*}
  \newcommand{\xarr}[1]{\xrightarrow{\makebox[5em]{$\scriptstyle #1$}}}
  \begin{aligned}
    \xi_X = \Big(
    M \otimes X
    & \xarr{\coev \otimes \id \otimes \id} X \otimes X^* \otimes M \otimes X \\[-3pt]
    & \xarr{\id \otimes \sigma \otimes \id} X \otimes M \otimes X^* \otimes X \\[-3pt]
    & \xarr{\id \otimes \id \otimes \iota_X} X \otimes M \otimes \mathbb{F}
    \xarr{\id \otimes a_M} X \otimes M \Big)
  \end{aligned}
\end{equation*}
for an object $X \in \mathcal{C}$. As Majid showed \cite[Theorem 3.2]{MR1192156}, the functor
\begin{equation}
  \label{eq:equiv-Maj}
  \mathcal{C}_{\mathbb{F}} \to \mathcal{Z}(\mathcal{C}),
  \quad (M, a_M) \mapsto (M, \xi)
\end{equation}
is an equivalence of categories. It turns out that the composition
\begin{equation*}
  \mathcal{C} \boxtimes \mathcal{C}
  \xrightarrow[\approx]{\quad \text{\eqref{eq:equiv-Lyu}} \quad}
  \mathcal{C}^{\mathbb{F}}
  \xrightarrow{\quad \omega^{\natural} \quad}
  \mathcal{C}_{\mathbb{F}}
  \xrightarrow[\approx]{\quad \text{\eqref{eq:equiv-Maj}} \quad}
  \mathcal{Z}(\mathcal{C}).
\end{equation*}
is isomorphic to the functor $G: \mathcal{C} \boxtimes \mathcal{C} \to \mathcal{Z}(\mathcal{C})$ of Definition~\ref{def:br-FTC-factorizable}. Hence $G$ is an equivalence if and only if $\omega^{\natural}$ is. This proves (1) $\Leftrightarrow$ (2) of Theorem~\ref{thm:non-degen}.

\subsection{Factorizability and the M\"uger center}

The proof of (2) $\Leftrightarrow$ (3) of Theorem~\ref{thm:non-degen} essentially depends on the theory of {\em Frobenius-Perron dimensions} for finite tensor categories. Let $\mathcal{C}$ be a finite tensor category, and let $\Gr(\mathcal{C})$ be the Grothendieck ring of $\mathcal{C}$. The {\em Frobenius-Perron dimension} is a unique ring homomorphism $\FPdim: \Gr(\mathcal{C}) \to \mathbb{R}$ such that $\FPdim(X) > 0$ for all non-zero objects $X \in \mathcal{C}$. We then define the Frobenius-Perron dimension of $\mathcal{C}$ by
\begin{equation*}
  \FPdim(\mathcal{C}) := \sum_{i \in I} \FPdim(P_i) \FPdim(V_i),
\end{equation*}
where $\{ V_i \}_{i \in I}$ is a complete set of representatives of isomorphism classes of simple objects of $\mathcal{C}$ and $P_i$ is a projective cover of $V_i$. See \cite{MR2119143} for details.

\begin{example}
  If $\mathcal{C} = \Rep(H)$ for some finite-dimensional (quasi-)Hopf algebra $H$, then $\FPdim(X) = \dim_k(X)$ for all objects $X \in \mathcal{C}$. By the representation theory of finite-dimensional algebras, we have $\FPdim(\mathcal{C}) = \dim_k(H)$.
\end{example}

Let $\mathcal{C}$ and $\mathcal{D}$ be finite tensor categories, and let $F: \mathcal{C} \to \mathcal{D}$ be a tensor functor ($=$ a $k$-linear exact strong monoidal functor). The {\em image} of $F$, $\Img(F)$, is the full subcategory of $\mathcal{D}$ consisting of all subquotients of objects of the form $F(X)$ for some $X \in \mathcal{C}$. We note that $\Img(F)$ is a tensor full subcategory of $\mathcal{D}$. The functor $F$ is said to be {\em dominant} (or {\em surjective} \cite{MR3242743}) if $\Img(F) = \mathcal{D}$. We now collect from \cite[Section 6.3]{MR3242743} basic results on the Frobenius-Perron dimension:
\begin{enumerate}
\item $F$ is fully faithful if and only if $\FPdim(\Img(F)) = \FPdim(\mathcal{C})$
\item $F$ is dominant if and only if $\FPdim(\Img(F)) = \FPdim(\mathcal{D})$.
\item Suppose that $F$ is either fully faithful or dominant. Then $F$ is an equivalence if and only if $\FPdim(\mathcal{C}) = \FPdim(\mathcal{D})$.
\end{enumerate}
Combining these results, we have the following criterion:
\begin{equation}
  \label{eq:equiv-criterion}
  \text{$F$ is an equivalence}
  \iff \FPdim(\mathcal{D}) = \FPdim(\Img(F)) = \FPdim(\mathcal{C}).
\end{equation}

We would like to apply~\eqref{eq:equiv-criterion} to the functor $G: \mathcal{C} \boxtimes \mathcal{C} \to \mathcal{Z}(\mathcal{C})$ used to define the factorizability. We consider a more general setting: Let $\mathcal{B}$ be a braided finite tensor category, and let $\mathcal{X}$ and $\mathcal{Y}$ be tensor full subcategories of $\mathcal{B}$. It is obvious that the intersection full subcategory $\mathcal{X} \cap \mathcal{Y}$ is a tensor full subcategory of $\mathcal{B}$. By the braiding of $\mathcal{B}$, we endow the functor
\begin{equation*}
  \mathcal{X} \boxtimes \mathcal{Y} \to \mathcal{B}, \quad X \boxtimes Y \mapsto X \otimes Y
\end{equation*}
with a structure of a tensor functor. We denote by $\mathcal{X} \vee \mathcal{Y}$ the image of this tensor functor. As a generalization of \cite[Lemma 8.21.6]{MR3242743} to the non-semisimple setting, we have:

\begin{lemma}[{\cite[Lemma 4.8]{2016arXiv160206534S}}]
  Under the above assumptions, we have
  \begin{equation*}
    \FPdim(\mathcal{X} \vee \mathcal{Y}) \FPdim(\mathcal{X} \cap \mathcal{Y})
    = \FPdim(\mathcal{X}) \FPdim(\mathcal{Y}).
  \end{equation*}
\end{lemma}

Now we can explain how to prove (2) $\Leftrightarrow$ (3) of Theorem~\ref{thm:non-degen}. Let $T_{+}$ and $T_{-}$ be as in Definition~\ref{def:br-FTC-factorizable}, and set $\mathcal{C}_{\pm} = \Img(T_{\pm})$. Since $T_{\pm}$ is fully faithful, we have
\begin{equation*}
  \FPdim(\mathcal{C}_{\pm}) = \FPdim(\mathcal{C}).
\end{equation*}
We also have $\mathcal{C}_{+} \vee \mathcal{C}_{-} = \Img(G)$. On the other hand, $\mathcal{C}_{+} \cap \mathcal{C}_{-}$ can be identified with the M\"uger center $\mathcal{C}'$. By \eqref{eq:equiv-criterion} and the well-known formula
\begin{equation*}
  \FPdim(\mathcal{C} \boxtimes \mathcal{C}) = \FPdim(\mathcal{Z}(\mathcal{C})) = \FPdim(\mathcal{C})^2,
\end{equation*}
we have the following logical equivalences:
\begin{align*}
  \text{$G$ is an equivalence}
  & \iff \FPdim(\mathcal{C}_{+} \vee \mathcal{C}_{-}) = \FPdim(\mathcal{C})^2 \\
  & \iff \FPdim(\mathcal{C}_{+} \cap \mathcal{C}_{-}) = 1 \\
  & \iff \FPdim(\mathcal{C}') = 1.
\end{align*}
Hence $\mathcal{C}$ is factorizable if and only if the M\"uger center of $\mathcal{C}$ is trivial.

\subsection{The M\"uger center and the map $\Phi_{\mathcal{C}}$}
\label{subsec:Mug-center}

We have explained the outlines of the proofs of (1) $\Leftrightarrow$ (2) and (2) $\Leftrightarrow$ (3) of Theorem~\ref{thm:non-degen}. The implications (1) $\Rightarrow$ (4) and (4) $\Rightarrow$ (5) are obvious. Thus it is sufficient to show (5) $\Rightarrow$ (3) to complete the proof of the theorem.

The proof of (5) $\Rightarrow$ (3) is based on the internal character theory \cite{MR3631720}, which extends results of the character theory of groups and Hopf algebras to finite tensor categories. Let, in general, $\mathcal{A}$ be a finite tensor category, and let $\mathcal{S}$ be a non-zero full subcategory of $\mathcal{A}$ closed under subquotients and direct sums. Then the coend
\begin{equation*}
  \mathbb{F}_{\mathcal{S}} := \int^{X \in \mathcal{S}} X^* \otimes X
\end{equation*}
exists (as an object of $\mathcal{A}$). By the universal property of this coend, we define the morphism $\phi_{\mathcal{S}}: \mathbb{F}_{\mathcal{S}} \to \mathbb{F}_{\mathcal{A}}$ to be the unique morphism in $\mathcal{A}$ such that $\phi_{\mathcal{S}} \circ \iota^{\mathcal{S}}_X = \iota^{\mathcal{A}}_X$ for all objects $X \in \mathcal{S}$, where $\iota^{\mathcal{A}}_X$ and $\iota^{\mathcal{S}}_X$ are the $X$-th component of the universal dinatural transformations of the coends $\mathbb{F}_{\mathcal{A}}$ and $\mathbb{F}_{\mathcal{S}}$, respectively. A key observation in \cite{MR3631720} is:

\begin{lemma}
  \label{lem:full-sub-cat-monic}
  The morphism $\phi_{\mathcal{S}}: \mathbb{F}_{\mathcal{S}} \to \mathbb{F}_{\mathcal{A}}$ is monic.
\end{lemma}

Set $\CF(\mathcal{A}) := \Hom_{\mathcal{A}}(\unitobj, \mathbb{F}_{\mathcal{A}})$ and $\Gr_k(\mathcal{A}) := k \otimes_{\mathbb{Z}} \Gr(\mathcal{A})$. We recall from Subsection~\ref{subsec:coend-F-property} that $\CF(\mathcal{A})$ is a $k$-algebra which can be thought of as a generalization of the algebra of class functions. Hence the number $j(\mathcal{A}) := \dim_k \CF(\mathcal{A})$ may be thought of as an analogue of the number of conjugacy classes of a finite group.

There is a tensor autoequivalence $(-)^{**}$ on $\mathcal{A}$ defined by taking the double dual object. A {\em pivotal structure} of $\mathcal{A}$ is an isomorphism $p: \id_{\mathcal{A}} \to (-)^{**}$ of tensor functors. Now we assume that $\mathcal{A}$ has a pivotal structure. Then the {\em internal character} of an object $X \in \mathcal{A}$ is defined by
\begin{equation}
  \label{eq:internal-char}
  \mathrm{ch}(X) := \Big( \unitobj
  \xrightarrow{\quad \coev_X \quad} X \otimes X^*
  \xrightarrow{\quad p_X \otimes \id_{X^*} \quad}
  X^{**} \otimes X^*
  \xrightarrow{\quad \iota^{\mathcal{A}}_{X^*} \quad} \mathbb{F}_{\mathcal{A}} \Big).
\end{equation}
Let $\{ V_i \}_{i \in I}$ be a complete set of representatives of isomorphism classes of $\mathcal{A}$, and set $\chi_i = \mathrm{ch}(V_i)$ for $i \in I$. One of main results of \cite{MR3631720} states that the set $\{ \chi_i \}_{i \in I}$ is linearly independent and the map
\begin{equation*}
  \mathrm{ch}: \Gr_k(\mathcal{A}) \to \CF(\mathcal{A}),
  \quad [X] \mapsto \mathrm{ch}(X)
\end{equation*}
is an injective algebra homomorphism. In particular, we have $j(\mathcal{A}) \ge |I|$. By using this inequality, we prove:

\begin{lemma}
  Let $\mathcal{A}$ be a pivotal finite tensor category. Then $j(\mathcal{A}) = 1$ if and only if $\mathcal{A} \approx \Vect$, the category of finite-dimensional vector spaces over $k$.
\end{lemma}
\begin{proof}
  The `if' part is clear. We prove the `only if' part. Suppose $j(\mathcal{A}) = 1$. Then, by the inequality $j(\mathcal{A}) \ge |I|$, every simple object of $\mathcal{A}$ is isomorphic to $\unitobj$. Hence $\mathcal{A}$ is unimodular (see Remark~\ref{rem:w-factor} for the definition) and therefore we have
  \begin{equation*}
    \dim_k \End(\id_{\mathcal{A}})
    = \dim_k \Hom_{\mathcal{A}}(\mathbb{F}_{\mathcal{A}}, \unitobj)
    = \dim_k \Hom_{\mathcal{A}}(\unitobj, \mathbb{F}_{\mathcal{A}})
    = j(\mathcal{A}) = 1.
  \end{equation*}
  As a generalization of the Maschke theorem for Hopf algebra, we have introduced a special non-zero element $\Lambda \in \End(\id_{\mathcal{A}})$, called an integral, such that $\Lambda^2 = 0$ if and only if $\mathcal{A}$ is semisimple \cite{MR3631720}. Since $\End(\id_{\mathcal{A}})$ is one-dimensional, the element $\Lambda$ cannot be nilpotent, and hence $\mathcal{A}$ is semisimple. Therefore $\mathcal{A} \approx \Vect$.
\end{proof}

\begin{remark}
  The proof is more technical than the above, but one can prove the same result for arbitrary finite tensor category: A finite tensor category $\mathcal{C}$ is equivalent to $\Vect$ if and only if $j(\mathcal{C}) = 1$ \cite{2016arXiv160206534S}.
\end{remark}

Now the implication (5) $\Rightarrow$ (3) of Theorem~\ref{thm:non-degen} is proved as follows: Let $\mathcal{C}$ be a braided finite tensor category, and let $\mathcal{C}'$ be the M\"uger center of $\mathcal{C}$. We note that, since the braiding restricted to $\mathcal{C}'$ is symmetric, $\mathcal{C}'$ admits a pivotal structure. We consider the composition
\begin{equation*}
  \Phi_{\mathcal{C}}' := \Big(
  \Hom_{\mathcal{C}}(\unitobj, \mathbb{F}_{\mathcal{C}'})
  \xrightarrow{\quad \Hom_{\mathcal{C}}(\unitobj, \phi_{\mathcal{C}'}) \quad}
  \Hom_{\mathcal{C}}(\unitobj, \mathbb{F}_{\mathcal{C}})
  \xrightarrow{\quad \Phi_{\mathcal{C}} \quad}
  \Hom_{\mathcal{C}}(\mathbb{F}_{\mathcal{C}}, \unitobj) \Big).
\end{equation*}
By the definition of $\Phi_{\mathcal{C}}$, it can be shown that $\mathrm{rank}(\Phi_{\mathcal{C}}') = 1$ (more precisely, the image of $\Phi_{\mathcal{C}}'$ is the subspace spanned by the counit $\varepsilon: \mathbb{F}_{\mathcal{C}} \to \unitobj$). Now we suppose that $\Phi_{\mathcal{C}}$ is injective. Then $\Phi_{\mathcal{C}}'$ is injective by Lemma~\ref{lem:full-sub-cat-monic}. Hence
\begin{equation*}
  j(\mathcal{C}')
  = \dim_k \Hom_{\mathcal{C}'}(\unitobj, \mathbb{F}_{\mathcal{C}'})
  = \dim_k \Hom_{\mathcal{C}}(\unitobj, \mathbb{F}_{\mathcal{C}'})
  = 1.
\end{equation*}
By the above lemma, we conclude that $\mathcal{C}' \approx \Vect$, that is, the M\"uger center of $\mathcal{C}$ is trivial. The proof of (5) $\Rightarrow$ (3) is done.

\section{Examples of modular tensor categories}
\label{sec:examples}

\subsection{The Drinfeld center}

Let $\mathcal{C}$ be a finite tensor category. Then $\mathcal{Z}(\mathcal{C})$ is a factorizable braided finite tensor category \cite{MR2097289}. Thus, by Theorem~\ref{thm:non-degen}, $\mathcal{Z}(\mathcal{C})$ is a modular tensor category provided that it admits a ribbon structure. For this reason, it is important to know when $\mathcal{Z}(\mathcal{C})$ admits a ribbon structure.

The pivotal structures and the ribbon structures of $\mathcal{Z}(\mathcal{C})$ are completely classified in \cite{2016arXiv160805905S} and \cite{2017arXiv170709691S}, respectively. Let $\alpha \in \mathcal{C}$ be the distinguished invertible object, and let $\delta_X: \alpha \otimes X \otimes \alpha^* \to X^{****}$ ($X \in \mathcal{C}$) be the Radford isomorphism introduced in \cite{MR2097289}. Following \cite{2017arXiv170709691S}, every ribbon structure of $\mathcal{Z}(\mathcal{C})$ is constructed from a pair $(\beta, j)$ consisting of an invertible object $\beta \in \mathcal{C}$ and an isomorphism $j_X: \beta \otimes X \otimes \beta^* \to X^{**}$ ($X \in \mathcal{C}$) of tensor functors such that $\beta \otimes \beta$ is isomorphic to $\alpha$ and the diagram
\begin{equation*}
  \xymatrix@C=56pt{
    \beta \otimes \beta \otimes X \otimes \beta^* \otimes \beta^*
    \ar[r]^(.57){j_{\beta \otimes X \otimes \beta^*}}
    & \beta \otimes X^{**} \otimes \beta^*
    \ar[r]^(.57){j_{X^{**}}}
    & X^{****} \\
    \ar@{=}[u] \ar[r]^(.57){\cong}
    \beta \otimes \beta \otimes X \otimes (\beta \otimes \beta)^*
    & \ar[r]^(.57){\delta_X}
    \alpha \otimes X \otimes \alpha^*
    & X^{****} \ar@{=}[u]
  }
\end{equation*}
commutes for all objects $X \in \mathcal{C}$. This result generalizes Kauffman and Radford's classification result of the ribbon elements of the Drinfeld double \cite{MR1231205}.

Douglas, Schommer-Pries and Snyder \cite{2013arXiv1312.7188D} introduced the notion of a {\em spherical} finite tensor category with motivation coming from local topological quantum field theories. Their definition coincides with Barrett-Westbury's \cite{MR1686423} in the semisimple case, but is different in general. A spherical structure in the sense of \cite{2013arXiv1312.7188D} is, in a word, a pivotal structure $p$ of $\mathcal{C}$ such that the pair $(\beta, j) = (\unitobj, p)$ makes the above diagram commute. Thus we have:

\begin{theorem}
  \label{thm:Dri-cen-mtc}
  The Drinfeld center of a finite tensor category $\mathcal{C}$ is a modular tensor category if $\mathcal{C}$ is spherical in the sense of \cite{2013arXiv1312.7188D}.
\end{theorem}

It is well-known that the Drinfeld center of a spherical fusion category (in characteristic zero) is a semisimple modular tensor category. Thus we have obtained a `non-semisimple' generalization of this fact, as expected in \cite[Section 6]{MR2681261}.

\subsection{The Yetter-Drinfeld category}

Let $\mathcal{C}$ be a braided finite tensor category, and let $B$ be a Hopf algebra in $\mathcal{C}$. Then one can define the braided finite tensor category ${}^B_B \mathcal{YD}(\mathcal{C})$ of Yetter-Drinfeld $B$-modules in $\mathcal{C}$ in a similar manner as the category of Yetter-Drinfeld modules over an ordinary Hopf algebra; see \cite{MR1456522}. We have proved in \cite{2016arXiv160206534S} that the M\"uger center of ${}^B_B \mathcal{YD}(\mathcal{C})$ is equivalent to the M\"uger center of $\mathcal{C}$. Hence, by Theorem~\ref{thm:non-degen}, ${}^B_B \mathcal{YD}(\mathcal{C})$ is a non-degenerate braided finite tensor category if and only if $\mathcal{C}$ is.

If $\mathcal{C} = \Rep(H)$ for some finite-dimensional quasitriangular Hopf algebra $H$, then ${}^B_B \mathcal{YD}(\mathcal{C})$ is the category of representations of the Hopf algebra obtained from $B$ and $H$ by the double-bosonization of Majid \cite{MR1645545}. Thus, if $H$ is factorizable, then the double-bosonization of $B$ by $H$ is again a factorizable Hopf algebra.

The above results say nothing about ribbon structures of ${}^B_B \mathcal{YD}(\mathcal{C})$. As remarked in \cite[Proposition 3.6.1]{MR1456522}, ${}^B_B \mathcal{YD}(\mathcal{C})$ is a braided tensor full subcategory of the Drinfeld center of the category ${}_{B}\mathcal{C}$ of left $B$-modules. Thus ${}^B_B \mathcal{YD}(\mathcal{C})$ admits a ribbon structure if $\mathcal{Z}({}_B\mathcal{C})$ is. By Theorem~\ref{thm:Dri-cen-mtc}, we have the following result:

\begin{theorem}
  Let $\mathcal{C}$ be a non-degenerate braided finite tensor category, and let $B$ be a Hopf algebra in $\mathcal{C}$. The Yetter-Drinfeld category ${}^B_B \mathcal{YD}(\mathcal{C})$ is a modular tensor category if ${}_B \mathcal{C}$ is spherical in the sense of \cite{2013arXiv1312.7188D}.
\end{theorem}

The author expects that this theorem could give an interesting example of modular tensor categories. Unfortunately, we have no useful criteria to check whether ${}_B \mathcal{C}$ is spherical. The sphericity of \cite{2013arXiv1312.7188D} is defined in terms the Radford isomorphism, but we do not know any description of the Radford isomorphism of ${}_B \mathcal{C}$ (although there has been closely related result \cite{MR3569179}).

\section{Categorical Verlinde formula}
\label{sec:verlinde-formula}

\subsection{The modular group action}

Let $\mathcal{C}$ be a modular tensor category with braiding $\sigma$ and twist $\theta$. Following Lyubashenko, we explain the projective representation of the modular group $\SL_2(\mathbb{Z})$ arising from $\mathcal{C}$.

We consider the Hopf algebra $\mathbb{F} = \mathbb{F}_{\mathcal{C}}$. The multiplication, the unit, the comultiplication, the counit, and the antipode of $\mathbb{F}$ are expressed by $m$, $u$, $\Delta$, $\varepsilon$ and $\gamma$, respectively. There is a non-zero morphism $\Lambda: \unitobj \to \mathbb{F}$ in $\mathcal{C}$ such that
\begin{equation*}
  m \circ (\id_{\mathbb{F}} \otimes \Lambda)
  = \Lambda \circ \varepsilon
  = m \circ (\Lambda \otimes \id_{\mathbb{F}}).
\end{equation*}
Such a morphism $\Lambda$ is called a (two-sided and $\unitobj$-based) {\em integral} of $\mathbb{F}$ (see \cite{MR1759389} and  \cite{MR1685417} for the general theory of integrals of braided Hopf algebras). We fix a non-zero integral $\Lambda$ satisfying the normalization condition
\begin{equation*}
  \omega \circ (\Lambda \otimes \Lambda) = \id_{\unitobj},
\end{equation*}
see, {\it e.g.}, \cite[Section 5.2.3]{MR1862634} for the existence of such an integral $\Lambda$. We also note that such an integral $\Lambda$ is unique up to sign.

Following \cite[Section 4]{MR1324034}, we define the {\em monodromy} $\Omega: \mathbb{F} \otimes \mathbb{F} \to \mathbb{F} \otimes \mathbb{F}$ to be the unique morphism in $\mathcal{C}$ such that the equation
\begin{equation*}
  \Omega \circ (\iota_X \otimes \iota_Y) = (\id_{X^*} \otimes \sigma_{Y^*,X}\sigma_{X,Y^*} \otimes \id_Y)
\end{equation*}
holds for all objects $X, Y \in \mathcal{C}$, where $\iota$ is the universal dinatural transformation for the coend $\mathbb{F}$. By using $\Lambda$ and $\Omega$, we define $S: \mathbb{F} \to \mathbb{F}$ by
\begin{equation*}
  S = (\varepsilon \otimes \id_{\mathbb{F}}) \circ \Omega \circ (\id_{\mathbb{F}} \otimes \Lambda).
\end{equation*}
We also define $T: \mathbb{F} \to \mathbb{F}$ by
\begin{equation*}
  T = \int^{X \in \mathcal{C}} \id_{X^*} \otimes \theta_X
  \quad \Big( = \int^{X \in \mathcal{C}} \theta_{X^*} \otimes \id_X \Big),
\end{equation*}
that is, the morphism such that the equation $T \circ \iota_X = \iota_X \circ (\id_{X^*} \otimes \theta_X)$ holds for all objects $X \in \mathcal{C}$.

Lyubashenko \cite{MR1324034} showed that the equations $(S T)^3 = \lambda S^2$, $S^2 = \gamma^{-1}$ and $S^4 = \theta_{\mathbb{F}}^{-1}$ hold in $\End_{\mathcal{C}}(\mathbb{F})$ up to scalar multiples. Since $\theta$ is a natural isomorphism, and since $\theta_{\unitobj} = \id_{\unitobj}$, we see that the assignment
\begin{equation*}
  \SL_2(\mathbb{Z}) \to \mathrm{GL}(\Hom_{\mathcal{C}}(\unitobj, \mathbb{F})),
  \quad \mathfrak{s} \mapsto \Hom_{\mathcal{C}}(\unitobj, S),
  \quad \mathfrak{t} \mapsto \Hom_{\mathcal{C}}(\unitobj, T)
\end{equation*}
defines a projective representation of $\SL_2(\mathbb{Z})$ on the space $\Hom_{\mathcal{C}}(\unitobj, \mathbb{F})$ of class functions of $\mathcal{C}$.

\subsection{The categorical Verlinde formula}

We now give the categorical Verlinde formula for a modular tensor category recently proposed by Gainutdinov and Runkel \cite{2016arXiv160504448G}.

Let $\mathcal{C}$ be a modular tensor category, and let $\{ V_i \}_{i \in I}$ be a complete set of representatives of isomorphism classes of simple objects of $\mathcal{C}$. Since a ribbon category has a pivotal structure, we can define the internal character $\mathrm{ch}(X)$ of an object $X \in \mathcal{C}$ by~\eqref{eq:internal-char}. Set $\CF(\mathcal{C}) := \Hom_{\mathcal{C}}(\unitobj, \mathbb{F})$, $\Gr_k(\mathcal{C}) := k \otimes_{\mathbb{Z}} \Gr(\mathcal{C})$ and $\chi_i := \mathrm{ch}(V_i)$ for $i \in I$. As we have recalled, the set $\{ \chi_i \}_{i \in I}$ is linearly independent and the map $\mathrm{ch}: \Gr_k(\mathcal{C}) \to \CF(\mathcal{C})$ is an injective homomorphism of algebras. Thus, if we have
\begin{equation*}
  [V_i \otimes V_j] = \sum_{a \in I} N_{i j}^a [V_a]
\end{equation*}
in the Grothendieck ring $\Gr(\mathcal{C})$, then
\begin{equation*}
  \chi_i \star \chi_j = \sum_{a \in I} N_{i j}^a \chi_a
\end{equation*}
in the algebra $\CF(\mathcal{C})$. This means that, in the case where $\mathrm{char}(k) = 0$, the `fusion rule' $\{ N_{i j}^a \}$ of $\mathcal{C}$ can be obtained by computing the product of $\chi_i$'s.

In view of applications to logarithmic conformal field theories, Gainutdinov and Runkel \cite{2016arXiv160504448G} utilize this fact to establish the categorical Verlinde formula. To explain their formula, we note that the integral $\Lambda$ induces an isomorphism
\begin{equation}
  \label{eq:cat-Ver-iso-1}
  \Hom_{\mathcal{C}}(\mathbb{F}, \unitobj) \to \Hom_{\mathcal{C}}(\unitobj, \mathbb{F}),
  \quad f \mapsto (f \otimes \id) \circ \Delta \circ \Lambda
\end{equation}
of vector spaces. For each $i \in I$, we define $\phi_i \in \End(\id_{\mathcal{C}})$ to be the element corresponding to $\chi_i$ via the isomorphism
\begin{equation}
  \label{eq:cat-Ver-iso-2}
  \Hom_{\mathcal{C}}(\unitobj, \mathbb{F})
  \xrightarrow{\quad \text{the inverse of \eqref{eq:cat-Ver-iso-1}} \quad}
  \Hom_{\mathcal{C}}(\mathbb{F}, \unitobj)
  \xrightarrow{\quad \eqref{eq:F-End-id} \quad}
  \End(\id_{\mathcal{C}}).
\end{equation}
Then the set $\{ \phi_i \}_{i \in I}$ is linearly independent in $\End(\id_{\mathcal{C}})$. Finally, we define the map $\mathfrak{S}_{\mathcal{C}}: \End(\id_{\mathcal{C}}) \to \End(\id_{\mathcal{C}})$ so that the following diagram is commutative:
\begin{equation*}
  \xymatrix@R=16pt@C=56pt{
    \Hom_{\mathcal{C}}(\unitobj, \mathbb{F})
    \ar[r]^{\text{\eqref{eq:cat-Ver-iso-2}}}
    \ar[d]_{\Hom_{\mathcal{C}}(\unitobj, S)}
    & \End(\id_{\mathcal{C}}) \ar[d]^{\mathfrak{S}_{\mathcal{C}}} \\
    \Hom_{\mathcal{C}}(\unitobj, \mathbb{F})
    \ar[r]^{\text{\eqref{eq:cat-Ver-iso-2}}}
    & \End(\id_{\mathcal{C}}).}
\end{equation*}

\begin{theorem}[Gainutdinov and Runkel \cite{2016arXiv160504448G}]
  \begin{equation}
    \label{eq:cat-Ver-GR}
    \mathfrak{S}_{\mathcal{C}}^{-1}(\mathfrak{S}_{\mathcal{C}}(\phi_i) \circ \mathfrak{S}_{\mathcal{C}}(\phi_j))
    = \sum_{a \in I} N_{i j}^a \phi_a
    \quad (i, j \in I).
  \end{equation}
\end{theorem}
\begin{proof}
  One can check that the diagram
  \begin{equation*}
  \xymatrix@R=16pt@C=56pt{
    \Hom_{\mathcal{C}}(\unitobj, \mathbb{F})
    \ar[r]^{\text{\eqref{eq:cat-Ver-iso-2}}}
    \ar[d]_{\Phi_{\mathcal{C}}}
    & \End(\id_{\mathcal{C}}) \ar[d]^{\mathfrak{S}_{\mathcal{C}}} \\
    \Hom_{\mathcal{C}}(\mathbb{F}, \unitobj)
    \ar[r]^{\text{\eqref{eq:F-End-id}}}
    & \End(\id_{\mathcal{C}}).}
  \end{equation*}
  is commutative. We recall from Subsection~\ref{subsec:coend-F-property} that both the maps $\Phi_{\mathcal{C}}$ and \eqref{eq:F-End-id} are homomorphisms of algebras. Hence the composition
  \begin{equation*}
    \Psi_{\mathcal{C}} := \Big( \Hom_{\mathcal{C}}(\unitobj, \mathbb{F})
    \xrightarrow{\quad \eqref{eq:cat-Ver-iso-2} \quad}
    \End(\id_{\mathcal{C}})
    \xrightarrow{\quad \mathfrak{S}_{\mathcal{C}} \quad}
    \End(\id_{\mathcal{C}}) \Big)
  \end{equation*}
  is a homomorphism of algebras (although neither \eqref{eq:cat-Ver-iso-2} nor $\mathfrak{S}_{\mathcal{C}}$ is in general). By the definition of $\phi_i$'s, we now have
  \begin{align*}
    \mathfrak{S}_{\mathcal{C}}(\phi_i) \circ \mathfrak{S}_{\mathcal{C}}(\phi_j)
    & = \Psi_{\mathcal{C}}(\chi_i) \circ \Psi_{\mathcal{C}}(\chi_j)
    = \Psi_{\mathcal{C}}(\chi_i \star \chi_j) \\
    & = \sum_{a \in I} N_{i j}^a \Psi_{\mathcal{C}}(\chi_a)
    = \sum_{a \in I} N_{i j}^a \mathfrak{S}_{\mathcal{C}}(\phi_a).
  \end{align*}
  The desired formula is obtained by applying $\mathfrak{S}_{\mathcal{C}}^{-1}$ to both sides.
\end{proof}

It is explained in \cite{2016arXiv160504448G} that this formula reduces to \eqref{eq:cat-Ver-ss} in the semisimple case (with an appropriate choice of the integral $\Lambda$). One also finds several applications of this formula and further results in \cite{2016arXiv160504448G,2017arXiv170201086F,2017arXiv170300150G,2017arXiv170608164F}.

\begin{remark}
  Given an object $V \in \mathcal{C}$, we define $\phi_V \in \End(\id_{\mathcal{C}})$ to be the element corresponding to the internal character $\mathrm{ch}(V) \in \Hom_{\mathcal{C}}(\unitobj, \mathbb{F})$ via \eqref{eq:cat-Ver-iso-2} and set $\psi_V = \mathfrak{S}_{\mathcal{C}}(\phi_V)$. Then the formula~\eqref{eq:cat-Ver-GR} is equivalent to
  \begin{equation}
    \label{eq:cat-Ver-GR-psi}
    \psi_{V_i} \circ \psi_{V_j} = \psi_{V_i \otimes V_j} \quad (i, j \in I).
  \end{equation}
  As explained in \cite[Remark 3.10 (2)]{2016arXiv160504448G}, for $X \in \mathcal{C}$, the $X$-th component of the natural transformation $\psi_{V}$ is given by the following composition:
  \begin{equation*}
    X \xrightarrow{\quad \coev' \otimes \id \quad}
    V^* \otimes V \otimes X \xrightarrow{\quad \id \otimes \sigma_{V,X}^{-1} \sigma_{X,V}^{-1} \quad}
    V^* \otimes V \otimes X \xrightarrow{\quad \eval \otimes \id \quad}
    X,
  \end{equation*}
  where $\coev' = (\id_{V^*} \otimes \theta_V) \circ \sigma_{V,V^*} \circ \coev_V$. Graphically, this morphism can be expressed as follows:
  \begin{equation*}
    (\psi_{V})_X = \quad \SelectTips{cm}{10} \xy /r1pc/:
    (0,1)="A0", "A0"-(.4,0)="A1"; "A0"+(.4,0)="A2"
    **\crv{
      "A1"+(-1,0) & "A1"+(-3,-.75) & "A1"+(-1,-1.5)
      & "A2"+(1,-1.5) & "A2"+(3,-.75) & "A2"+(1,0)}
    ?> *\dir{>} ?(.75) *+!L{V} ?(.5)="P1",
    (0,2); "P1"+(0,.5) **\dir{-} ?< *+!D{X},
    "P1"-(0,.5); (0,-2) **\dir{-} ?> *\dir{>},
    \endxy
  \end{equation*}
  Once we have obtained this graphical expression, the formula \eqref{eq:cat-Ver-GR-psi} can easily proved by the graphical calculus in a ribbon category.
\end{remark}

\begin{remark}
  If $\mathcal{C}$ is semisimple, then we have
  \begin{equation*}
    (\phi_i)_{V_j} = \delta_{i j} \frac{\sqrt{\dim(\mathcal{C})}}{\dim(V_i)} \, \id_{V_j}
  \end{equation*}
  for all $i, j \in I$, where the square root of $\dim(\mathcal{C})$ depends on the choice of $\Lambda$ \cite[Remark 3.10 (3)]{2016arXiv160504448G}. One also finds several results on $\phi_i$'s in the non-semisimple case in \cite{2017arXiv170300150G}. In that paper, they also established another version of the categorical Verlinde formula that generalizes the result of Cohen-Westreich \cite{MR2464107}.
\end{remark}

\bibliographystyle{alpha}
\def\cprime{$'$}

\end{document}